\pdfoutput=1
\documentclass[a4paper,12pt]{article}
\usepackage{ifpdf}
\ifpdf
\usepackage[pdftex]{graphicx}
\usepackage[usenames,dvipsnames]{color}
\else
\usepackage{graphicx}
\usepackage[usenames,dvips]{color}
\fi
\usepackage{amsmath}
\usepackage{amssymb}
\usepackage{amsthm}
\usepackage[colorlinks=true]{hyperref}   
\usepackage{caption}
\usepackage{paper}

\numberwithin{equation}{section}

\begin{document}

\title{Curvature corrected estimates for geodesic arc-length.}
\author{%
Leo Brewin\\[10pt]%
School of Mathematical Sciences\\%
Monash University, 3800\\%
Australia}
\date{14-Oct-2015}
\date{19-Oct-2015}
\reference{Preprint}

\maketitle

\begin{abstract}
\noindent
We will develop simple relations between the arc-lengths of a pair of geodesics that share
common end-points. The two geodesics differ only by the requirement that one is constrained to
lie in a subspace of the parent manifold. We will present two applications of our results. In
the first example we explore the convergence of Gaussian curvature estimates on a simple
triangular mesh. The second example demonstrates an improved error estimate for the area of a
Schwarz lantern.

\end{abstract}

\section{Introduction}
\label{sec:intro}
If $S$ is an $(n-1)$-dimensional subspace of an $n$-dimensional Riemannian space $M$ then $S$
will inherit a metric from its embedding in $M$. Consider a pair of points $p$ and $q$ on $S$
chosen sufficiently close to ensure that there exists a unique geodesic in $S$ connecting $p$
to $q$. If necessary, the points $p$ and $q$ could be further constrained to ensure that a
unique geodesic, this time lying in $M$, also connects $p$ to $q$. A natural question to pose
would be -- how are the arc-lengths of the two geodesics related? This is a simple question
and as shown below rather easy to answer. We will show that the difference between the
arc-lengths is (not surprisingly) controlled by the embedding of $S$ in $M$, that is, the
difference can be expressed solely in terms of the second fundamental form of $S$.

\section{Geodesic arc-length}
\label{sec:geodesiclsq}

\begin{lemma*}
Let $M$ be an $n$-dimensional manifold with a Riemannian metric $g$. Consider an
$(n-1)$-dimensional subspace $S$ of $M$ and let $h$ be the metric induced on $S$ by the
embedding of $S$ in $(M,g)$. Choose a set of coordinates on $M$ and choose the natural
coordinate basis $\partial_\mu$ for the tangent space on points of $M$. Then
\begin{equation}
   \label{eqn:3d4dconn}
   {\tilde\Gamma}^\mu{}_{\alpha\beta}
 = {\bar\Gamma}^\mu{}_{\alpha\beta} + n^\mu K_{\alpha\beta}
\end{equation}
where ${\tilde\Gamma}^\mu{}_{\alpha\beta}$ and ${\bar\Gamma}^\mu{}_{\alpha\beta}$ are the
metric compatible connection components on $S$ and $M$ respectively and where
$K_{\alpha\beta}$ are the components of the extrinsic curvature of $S$.
\end{lemma*}
\begin{proof}
Let $v^\mu$ be the tangent vector to a geodesic of $(S,h)$ and let $\vert$ denote the
covariant derivative with respect to ${\tilde\Gamma}^\mu{}_{\alpha\beta}$. Then
\begin{equation}
   \label{eqn:3dgeod}
   0 = v^\mu{}_{\vert \nu} v^\nu
\end{equation}
However
\begin{equation}
   \label{eqn:3dperp}
   v^\mu{}_{\vert\nu}
   = \bot\left(v^\mu{}_{;\nu}\right)
   = \left(\delta^\mu{}_{\nu} + n^\mu n^\nu\right)\left(v^\mu{}_{;\nu}\right)
\end{equation}
where $;$ denotes the covariant derivative with respect to
${\bar\Gamma}^\mu{}_{\alpha\beta}$. Substituting (\ref{eqn:3dperp}) into (\ref{eqn:3dgeod})
while also using $0=(n_\alpha v^\alpha)_{;\beta}$ and
$2K_{\alpha\beta} = \bot(n_{\alpha;\beta}+n_{\beta;\alpha})$ leads to
\begin{equation}
   0 = v^\mu{}_{;\nu}v^\nu + n^\mu K_{\alpha\beta} v^\alpha v^\beta
\end{equation}
Comparing this equation with (\ref{eqn:3dgeod}) leads directly to (\ref{eqn:3d4dconn}).
\end{proof}
Note that ${\tilde\Gamma}^\mu{}_{\alpha\beta}$ can be viewed as a second connection on $M$
constructed so that it admits geodesics lying solely within $S$.

\begin{theorem*}
Let $M$ be an $n$-dimensional manifold with a Riemannian metric $g$. Consider an
$(n-1)$-dimensional subspace $S$ of $M$ and let $h$ be the metric induced on $S$ by the
embedding of $S$ in $(M,g)$. Consider a pair of nearby points $p$ and $q$ in $S$ chosen so
that they are connected by a pair of unique geodesics, one in $(S,h)$ and the other in
$(M,g)$. Let ${\bar L}(p,q)$ be the arc-length from $p$ to $q$ in $(M,g)$ and likewise let
${\tilde L}(p,q)$ be the arc-length in $(S,h)$. Let $K$ be the second fundamental form for
$S$ and $v$ the unit-tanget to the geodesic in $(M,g)$. Then
\begin{equation}
\label{eqn:main}
{\tilde L}^2(p,q) = {\bar L}^2(p,q)
                  + \frac {1}{12}\left(K(v,v){\bar L}^2(p,q)\right)^2
                  + \BigO{{\bar L}^5}
\end{equation}
\end{theorem*}

\begin{proof}
We begin by constructing a local set of Riemann normal coordinates $x^\mu$ that covers a
subset of $M$ containing the two points $p$ and $q$. The assumption that there is a unique
geodesic in $M$ that connects $p$ and $q$ ensures that such a set of coordinates can be
constructed. We are free to locate the origin of the coordinates to be at $p$ and also to
align the coordinates axes so that one axis is parallel to the geodesic connecting $p$ to
$q$. Let that axis be the $x^1$ axis. Then the coordinates of $p$ are $x_p=(0,0,0...)$ while
for $q$ we have $x_q=(x^1_q,0,0,...)$. The $x^\mu$ do not provide a set of Riemann normal
coordinates on $S$. However in a new set of coordinates $y^\mu$ in the neighbourhood of $p$
given by
\begin{equation}
   \label{eqn:3drncA}
   y^\mu = x^\mu
         + \frac{1}{2}{\tilde\Gamma}^\mu{}_{\alpha\beta} x^\alpha x^\beta
         + \frac{1}{6}\left({\tilde\Gamma}^\mu{}_{\alpha\beta}
                            {\tilde\Gamma}^\beta{}_{\theta\phi}
                           +{\tilde\Gamma}^\mu{}_{\theta\phi,\alpha}
                       \right) x^\alpha x^\theta x^\phi
         + \BigO{L^4}
\end{equation}
it is easy to verify that the connection components at $p$ satisfy
\begin{align}
   0 &= {\Gamma}^\mu{}_{\alpha\beta}\label{eqn:3drncB}\\
   0 &= {\Gamma}^\mu{}_{(\alpha\beta,\rho)\label{eqn:3drncC}}
\end{align}
and thus the $y^\mu$, when restricted to $S$, serve as a set of Riemann normal coordinates
on $S$. In the $y^\mu$ coordinates we have $y_p=(0,0,0\cdots)$ and $y_q=(y^1_q,0,0,...)$.
The choice of Riemann normal coordinates is motived by the following simple expression for
the geodesic arc-length
\begin{equation}
   \label{eqn:lsqGeod}
   L^2(p,q) = g_{\mu\nu} \Delta x^\mu_{pq}\Delta x^\nu_{pq}
            - \frac{1}{3} R_{\mu\alpha\nu\beta} x^\mu_p x^\nu_p x^\alpha_q x^\alpha_q
            + \BigO{L^5}
\end{equation}
where $g_{\mu\nu}=\diag(1,1,1,\cdots)$ and $R_{\mu\alpha\nu\beta}$ are the Riemann curvature
components evaluated at $p$ and where $\Delta x^\mu_{pq} = x^\mu_q - x^\mu_p$. There are two
geodesics to be considered. They both join $p$ to $q$ but one uses the connection
${\tilde\Gamma}^\mu{}_{\alpha\beta}$ while the other uses
${\bar\Gamma}^\mu{}_{\alpha\beta}$. The squared arc-lengths for this pair of geodesics can
be found using (\ref{eqn:lsqGeod}) leading to
\begin{align}
   {\tilde L}^2(p,q) &= g_{\mu\nu} y^\mu_q y^\nu_q\label{eqn:3dlsq}\\
   {\bar L}^2(p,q)   &= g_{\mu\nu} x^\mu_q x^\nu_q\label{eqn:4dlsq}
\end{align}
Now we can combine equations (\ref{eqn:3d4dconn}) and (\ref{eqn:3drncA}) and substitute the
result into (\ref{eqn:3dlsq}) to obtain
\begin{equation}
   {\bar L}^2(p,q) = {\tilde L}^2(p,q)
                   - \frac{1}{4} \left(K_{\mu\nu} x^\mu_q x^\nu_q\right)^2
                   + \frac{1}{3} K_{\mu\nu}K_{\alpha\beta}
                                 x^\mu_q x^\nu_q x^\alpha_q x^\beta_q
                   + \BigO{L^5}
\end{equation}
Now recall that $x_q=(x^1_q,0,0,0\cdots)$ and thus $x_q = {\tilde L}(p,q) v$ where $v$ is a
unit vector from $p$ to $q$. Thus we can re-write the previous equation as
\begin{equation}
   {\bar L}^2(p,q) = {\tilde L}^2(p,q)
                   + \frac{1}{12}\left(K_{\mu\nu} v^\mu v^\nu {\tilde L}^2(p,q)\right)^2
                   + \BigO{L^5}
\end{equation}
which completes the proof.
\end{proof}

\begin{corollary}
   Let $n(p)$ and $n(q)$ be the unit normal vectors on $S$ at points $p$ and $q$
   respectively. Then
   \begin{equation}
      \label{eqn:corollary1}
      {\bar L}^2(p,q) = {\tilde L}^2(p,q)
                      + \frac{1}{12}\left( (n_\mu(p)-n_\mu(q))\Delta x^\mu_{pq}\right)^2
                      + \BigO{L^5}
   \end{equation}
\end{corollary}
\begin{proof}
   Since the connection ${\bar\Gamma}^\mu{}_{\alpha\beta}$ vanishes at $p$ we have
   \begin{equation}
      (n_\mu(p)-n_\mu(q))v^\mu = n_{\mu;\nu}\Delta x^\nu_{pq} v^\mu + \BigO{L^2}
   \end{equation}
   But $\Delta x^\nu_{pq} = v^\nu {\bar L(p,q)}$ thus we also have
   \begin{equation}
      (n_\mu(p)-n_\mu(q))v^\mu = n_{\mu;\nu} v^\mu v^\nu {\bar L(p,q)} + \BigO{L^2}
   \end{equation}
   The vector $v^\mu$, which is tangent to the geodesic in $(M,g)$ connecting $p$ to $q$, is
   in general not tangent to the geodesic in $(S,h)$. It can however be written as a linear
   combination of a vectors parallel and perpendicular to $S$ at $p$. That is
   \begin{equation}
      v^\mu = \alpha {\tilde v}^\mu + \beta n^\mu
   \end{equation}
   where $\alpha$ and $\beta$ are numbers yet to be determined and ${\tilde v}^\mu$ is the
   unit tangent vector to the geodesic in $S$. It is clear that when $p$ and $q$ are close
   then $\alpha=1 + \BigO{L^2}$ and $\beta = \BigO{L}$. Substituting this into the previous
   equation leads to
   \begin{equation}
      (n_\mu(p)-n_\mu(q))v^\mu = n_{\mu;\nu} {\tilde v}^\mu {\tilde v}^\nu {\bar L(p,q)}
                               + \BigO{L^2}
   \end{equation}
   Now we can use
   \begin{equation}
      2K_{\mu\nu} = \bot\left(n_{\mu;\nu}+n_{\mu;\nu}\right)
   \end{equation}
   to obtain
   \begin{equation}
      (n_\mu(p)-n_\mu(q))v^\mu = K_{\mu\nu} {\tilde v}^\mu {\tilde v}^\nu {\bar L(p,q)}
                               + \BigO{L^2}
   \end{equation}
   Multipling through by ${\bar L}(p,q)$ and using $\Delta x^\mu_{pq} = v^\mu {\bar L}(p,q)$
   we see that (\ref{eqn:main}) can be written as
   \begin{equation}
    {\bar L}^2(p,q) = {\tilde L}^2(p,q)
                    + \frac{1}{12}\left( (n_\mu(p)-n_\mu(q))\Delta x^\mu_{pq}\right)^2
                    + \BigO{L^5}
   \end{equation}
   which completes the proof.
\end{proof}

\begin{corollary}
   Consider a one parameter family of hypersurfaces generated from $S$ by dragging $S$ along
   its unit normal $n$. This family forms a local foliation of $M$ in which the points $p$
   and $q$ are now viewed as functions along the integral curves of $n$. Then
   \begin{equation}
      \label{eqn:corollary2}
      {\bar L}^2(p,q) = {\tilde L}^2(p,q)
                      + \frac{1}{48}\left( \frac{d{\bar L}(p,q)}{dn}\right)^4
                      + \BigO{L^5}
   \end{equation}
   where $n$ (not to be confused with the unit normal) is the arc length measured along the
   integral curves of the unit normal.
\end{corollary}
\begin{proof}
   The equation for the first variation of arc states that
   \begin{equation}
      \frac{d{\bar L}(p,q)}{dn} = \left[n_\mu v^\mu\right]_p^q
   \end{equation}
   where $v^\mu$ is the unit tangent vector to the geodesic in $(M,g)$. But in our Riemann
   normal coordinates we have $v^\mu_p = v^\mu_q$ and thus we also have
   \begin{equation}
      \frac{d{\bar L}(p,q)}{dn} = \left(n_\mu(p)-n_\mu(q)\right) v^\mu
   \end{equation}
   which leads to
   \begin{equation}
      \frac{d{\bar L}^2(p,q)}{dn} = 2\left(n_\mu(p)-n_\mu(q)\right) \Delta x^\mu_{pq}
   \end{equation}
   Combining this with the previous corollary completes the proof.
\end{proof}

\section{Examples}
\label{sec:examples}

\subsection{Estimating Gaussian curvature}
It is common practice in computer graphics to model a smooth 2-dimensional surfaces such as
a sphere, a torus or even teapots by a finite collection of connected triangles. The
vertices of the triangles are taken as sample points of the smooth surface while the legs
are taken as geodesics of the flat 3-dimensional space in which the surface resides. This
discrete approximation to the smooth surface is commonly known as a triangulation.

One of the more important quantities associated with any 2-dimensional surface is the
Gaussian curvature. This is usually computed by taking various derivatives on a smooth
surface. Yet that is clearly not possible on a triangulation (as a smooth function) since
the local metric is at best piecewise constant. Nonetheless it seems reasonable to expect
that where a triangulation closely approximates a smooth surface then the curvature on the
triangulation should be close to the curvature of the smooth surface. How then can such a
curvature on a triangulation be computed? Various methods
(\cite{%
surazhsky:2003-01,%
petitjean:2002-01,%
meyer:2003-01,%
magid:2007-01,%
liu:2007-01,%
borrelli:2003-01%
})
have been developed over the years that broadly speaking divide into two approaches. In one
approach a smooth surface is interpolated through the vertices which in turn allows the
curvature to be computed using standard methods (see \cite{petitjean:2002-01} for an
extensive review). The other approach uses area weighted sums to estimate the local
curvature (see \cite{meyer:2003-01}).

It is well known that for the case where 4 triangles meet at a vertex the estimated Gaussian
curvature need not converge to the correct value (as the triangulation is refined towards a
continuum limit). Here we shall demonstrate that failure for the simple case of four
identical triangles on a 2-sphere. We will also show that the correct convergent estimate of
the curvature can be recovered by using an adjusted set of leg lengths given by the main
theorem.

The Gaussian curvature on a unit 2-sphere $S^2$ in $E^3$ is 1 everywhere on $S^2$. Consider
now any point $p$ on $S^2$ enclosed by 4 equally spaced points $a$, $b$, $c$ and $d$ also on
$S^2$. This set of points can be connected to form 4 triangles attached to $p$ as indicated
in figure (\ref{fig:sphere}). The Gaussian curvature at $p$ will be estimated by solving the
coupled system of equations
\begin{equation}
   \label{eqn:rncLsq}
   L^2_{ij} = g_{\mu\nu} \Delta x^\mu_{ij} \Delta x^\nu_{ij}
            - \frac {1}{3} R_{\alpha\mu\beta\nu} x^\alpha_i x^\beta_j x^\mu_i x^\nu_j
\end{equation}
for the Riemann normal coordinates $x^\mu_i$ for each vertex $i=p,a,b,c,d$ and the Riemann
components $R_{\alpha\mu\beta\nu}$ at $p$. The $L^2_{ij}$ are the squared arc-lengths
between vertices $i$ and $j$. In the first instance we will take the $L_{ij}$ to be the
Euclidian arc-length given by the embedding of the vertices in $E^3$. Later we will adjust
the $L_{ij}$ by using equation (\ref{eqn:corollary2}).

The symmetry of the 2-sphere allows us to choose all triangles to be identical and to also
choose the Cartesian and Riemann coordinates of each vertex as per table
(\ref{tbl:SphereCoords}). With this choice of coordinates the equations (\ref{eqn:rncLsq})
can be reduced to just two equations, namely
\begin{align}
   {\bar L}^2_{pa} &= {\bar x}^2 + \left({\bar z}-1\right)^2
                    = {\tilde x}^2\label{eqn:sphereXZa}\\
   {\bar L}^2_{ab} &= 2{\bar x}^2
                    = 2{\tilde x}^2 - \frac{1}{3} K {\tilde x}^4\label{eqn:sphereXZb}
\end{align}
where $K=R_{1212}$ is the Gaussian curvature at $p$. The constraint that the points lie on
the unit sphere leads to just one equation
\begin{equation}
   1 = {\bar x}^2 + {\bar z}^2
\end{equation}
Thus we have three equations for four unknowns ${\bar x},\>{\bar z},\>{\tilde x}$ and $K$.
Clearly we can choose $\bar x$ as a free parameter which leads to the following solution for
$K$
\begin{equation}
   K = \frac{3}{2} + \BigO{{\bar x}^2}
\end{equation}
This shows clearly that as the set of points converge to $p$ the estimate for $K$ converges
to the incorrect value of $3/2$. This is a well known result and is not germain to the use
of Riemann normal coordinates (see \cite{xu-xu:2009-01,hildebrandt:2006-01}).

\bgroup
\def\H{\vrule height 14pt depth  7pt width 0pt}
\def\m{\vrule height  0pt depth 10pt width 0pt}
\def\M{\vrule height 15pt depth 10pt width 0pt}
\def\A#1{\hbox to 17pt{\hfill$#1$}}
\def\Z{\A{0}}
\def\U{\A{1}}
\def\tX{\A{{\tilde x}}}
\def\mtX{\A{-{\tilde x}}}
\def\bX{\A{{\bar x}}}
\def\mbX{\A{-{\bar x}}}
\def\z{\A{\bar z}}
\begin{table}[ht]
\begin{center}
\begin{tabular}{ccccccc}
\hline
\H&Vertex&&\multicolumn{3}{c}{Coordinates}&\\
&&&Cartesian&&Riemann&\\
\hline
\M&$p$&&$(\Z,\Z,\U)$    &&$(\Z,\Z)$      \\
\m&$a$&&$(\bX,\Z,\z)$   &&$(\tX,\Z)$      \\
\m&$b$&&$(\Z,\bX,\z)$   &&$(\Z,\tX)$      \\
\m&$c$&&$(\mbX,\Z,\z)$  &&$(\mtX,\Z)$     \\
\m&$d$&&$(\Z,\mbX,\z)$  &&$(\Z,\mtX)$     \\
\hline
\end{tabular}
\end{center}
\caption{The Cartesian and Riemann normal coordinates of the 5 vertices. This choices makes
full use of the known symmetries of the 2-sphere. The size of the triangles is controlled by
the freely chosen coordinate $\bar x$ while $\bar z$ is set by the contraint that the points
lie on the unit 2-sphere. The Riemann normal coordinate $\tilde x$ and the Gaussian
curvature can then be computed from the leg-lengths as described in the text.}
\label{tbl:SphereCoords}
\end{table}
\egroup

\begin{figure}[ht]
\Figure{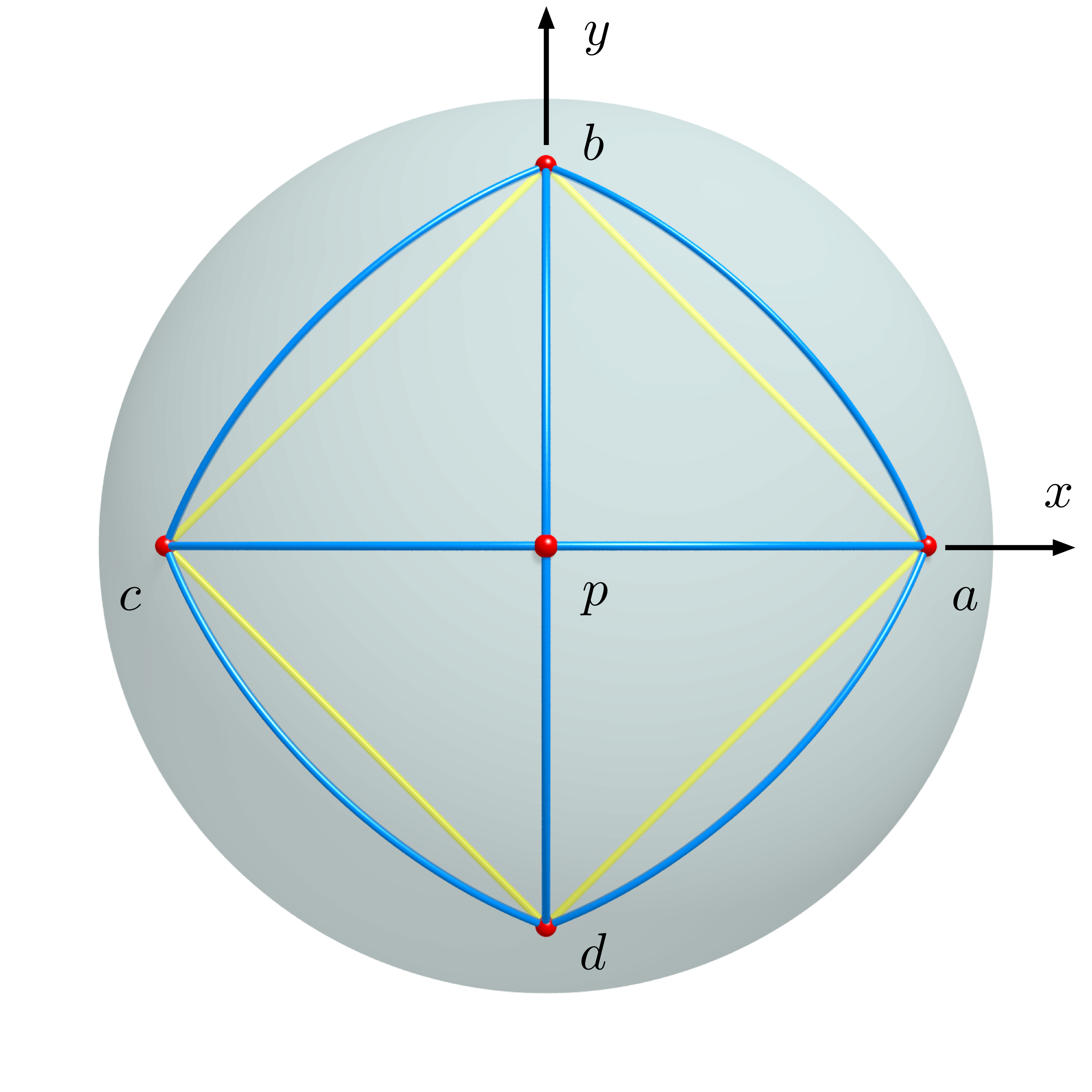}{0.80}
\caption{\normalfont%
A (not so small) patch of triangles on a unit 2-sphere. This is a view looking down the
$z-$axis onto the north pole. The vertices $p,a,b,c$ and $d$ lie on the the various
coordinate planes with Cartesian and Riemann normal coordinates as listed in table
(\ref{tbl:SphereCoords}). The yellow edges are the geodesics in $E^3$ while the blue edges
are the geodesics on the sphere.}
\label{fig:sphere}
\end{figure}

We now repeat the computations but this time using estimates for the geodesic arc-length on
the unit sphere as given by equation (\ref{eqn:corollary2}). For the unit sphere it is easy
to see that
\begin{equation}
   \frac{dL}{dn} = L
\end{equation}
and thus from (\ref{eqn:corollary2}) we find
\begin{equation}
   {\tilde L}^2_{ij} = {\bar L}^2_{ij} + \frac{1}{12}{\bar L}^4_{ij} + \BigO{L^5}
\end{equation}
Using this equation to estimate ${\tilde L}^2_{pa}$ and ${\tilde L}^2_{ab}$ leads to the
following adjusted Riemann normal equations
\begin{align}
   {\tilde L}^2_{pa} &= {\bar L}^2_{pa} + \frac{1}{12} {\bar L}^4_{pa} = {\tilde x}^2\\[5pt]
   {\tilde L}^2_{ab} &= {\bar L}^2_{ab} + \frac{1}{12} {\bar L}^4_{ab} = 2{\tilde x}^2 -
                                                            \frac{1}{3} K {\tilde x}^4
\end{align}
where ${\bar L}^2_{pa}$ and ${\bar L}^2_{ab}$ should be considered as functions of ${\bar
x}$ and ${\bar z}$ given by equations (\ref{eqn:sphereXZa},\ref{eqn:sphereXZb}). Once again
we have three equations for four unknowns. This system can be solved in exactly the same
manner as before to obtain
\begin{equation}
   \label{eqn:betterK}
   K = 1 + \BigO{{\bar x}^2}
\end{equation}
which clearly gives the correct result as ${\bar x}\rightarrow0$.

A reasonable objection to this approach is that in obtaining equation (\ref{eqn:betterK}) we
have made use of known properties of the unit sphere. Thus it should be no surprise that we
get a better result. However we could easily propose a hybrid scheme in which the method
described by Meyer \etal (\cite{meyer:2003-01}) would be used to estimate the normals at the
vertices which in turn would allow us to use equation (\ref{eqn:corollary1}) to estimate
${\tilde L}^2_{ij}$. Numerical experiments on such a hybrid scheme indicates that it can
offer improvements over standard methods. The results will be reported elsewhere.

\subsection{The Schwarz lantern}

There are many ways to triangulate a cylinder such as the Schwarz lantern shown in figure
(\ref{fig:lantern}). This particular triangulation was chosen by Schwarz (see
\cite{morvan:2008-01,dubnov:1963-01}) to provide a simple counter example to a claim that if
all the points of a triangulation converge to a smooth surface (i.e., by creating by more and
more triangles while decreasing their size to zero) then the surface area of the
triangulation would converge to the area of the surface.

We will do the standard computation that establishes error bounds for the area. We will then
repeat the computation but this time using the adjusted arc-lengths given by
(\ref{eqn:corollary1}) yielding an improved estimate for the error bounds.

\begin{figure}[ht]
\Figure{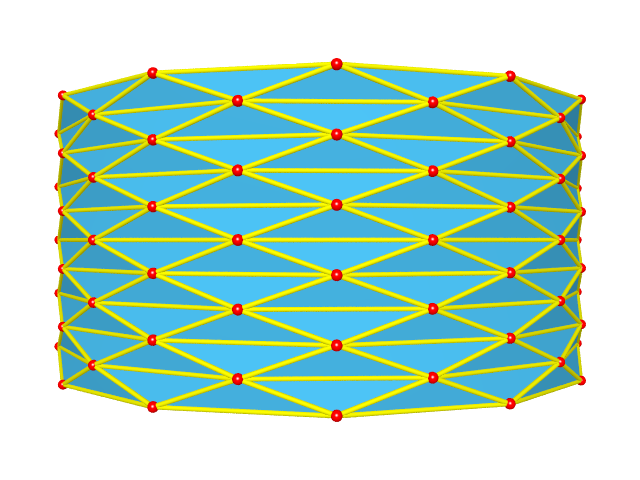}{0.90}
\caption{\normalfont%
An example of a Schwarz lantern built from 220 identical triangles. In the general case the
cylinder is divided into $2M$ horizontal slices and $2N$ vertical slices for a total of
$4NM$ triangles. The example shown here has $N=11$ and $M=5$.}
\label{fig:lantern}
\end{figure}

We begin by first orienting the Cartesian coordinate axes so that the $z-$axis runs up the
centre of the cylinder (see figure (\ref{fig:cylinder})) while the $x-$axis passes through a
vertex $p$ of a typical triangle (recall that all triangles are identical to each other
modulo reflections in the $xy-$plane). The coordinates of the three vertices are shown in
table (\ref{tbl:LanternCoords}). Using standard Euclidian geometry it is easy to show that
\begin{align}
   {\bar L}^2_{pr} &=4\sin^2\left(\frac{\pi}{N}\right)\\[5pt]
   {\bar L}^2_{pq} &=4\sin^2\left(\frac{\pi}{2N}\right) + \frac{1}{4M^2}\\[5pt]
   A^2_{pqr} &= \left(\frac{\pi}{2NM}\right)^2\\[5pt]
   {\bar A}^2_{pqr} &= \frac {1}{16}{\bar L}^2_{pr}\left(4{\bar L}^2_{pq}
                                                         -{\bar L}^2_{pr}\right)
\end{align}
where $A_{pqr}$ is the exact area (i.e., the area of a triangle drawn entirely on the
cylinder) and ${\bar A}_{pqr}$ is the area of the flat triangle with vertices $p,\>q$ and
$r$.

\bgroup
\def\H{\vrule height 14pt depth  7pt width 0pt}
\def\m{\vrule height  0pt depth 10pt width 0pt}
\def\M{\vrule height 15pt depth 10pt width 0pt}
\def\C#1{\hbox to 20pt{\hfill$#1$}}
\def\B#1{\hbox to 35pt{\hfill$#1$}}
\def\Px{\B{1}}
\def\Py{\B{0}}
\def\Pz{\C{0}}
\def\Qx{\B{\cos\frac{\pi}{N}}}
\def\Qy{\B{\sin\frac{\pi}{N}}}
\def\Qz{\C{\frac{1}{2M}}}
\def\Rx{\B{\cos\frac{2\pi}{N}}}
\def\Ry{\B{\sin\frac{2\pi}{N}}}
\def\Rz{\C{0}}
\begin{table}[t]
\begin{center}
\begin{tabular}{ccccc}
\hline
\H&Vertex&&Cartesian Coordinates&\\
\hline
\M&$p$&&$(\Px,\Py,\Pz)$      \\
\m&$q$&&$(\Qx,\Qy,\Qz)$      \\
\m&$r$&&$(\Rx,\Ry,\Rz)$      \\
\hline
\end{tabular}
\end{center}
\caption{The Cartesian coordinates of the 3 vertices on the Schwarz lantern.}
\label{tbl:LanternCoords}
\end{table}
\egroup

\begin{figure}[ht]
\Figure{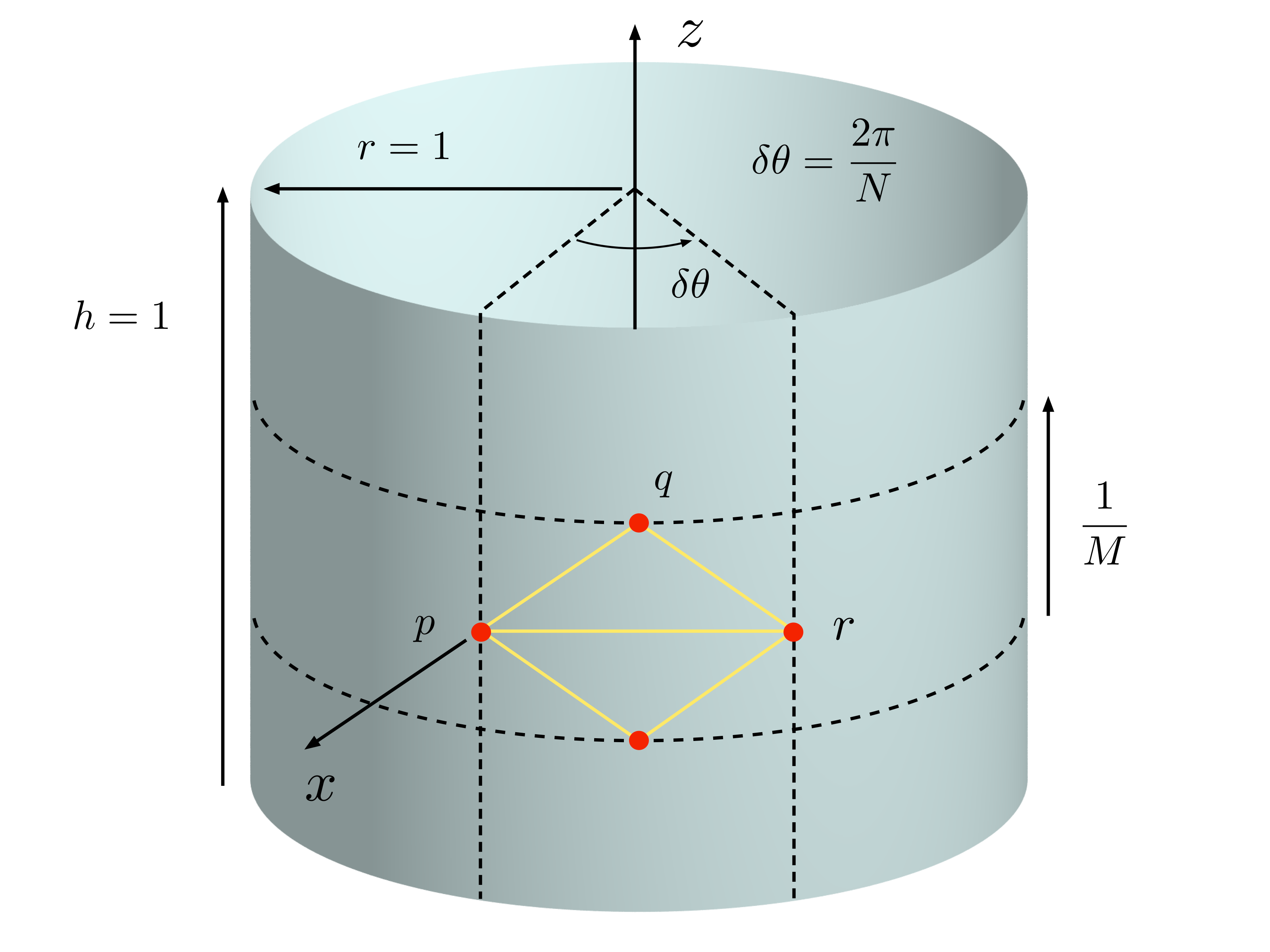}{0.90}
\caption{\normalfont%
A typical pair of triangles in the Schwarz lantern. In the text we compute the area
$A_{pqr}$ of the triangle based on the vertices $p,q$and $r$. The total area of the smooth
cylinder is $S=2\pi$ while the sume of the areas over the triangulation is $4NMA_{pqr}$.}
\label{fig:cylinder}
\end{figure}

It is not hard to show the fractional error in using ${\bar A}^2_{pqr}$ as an approximation
to $A^2_{pqr}$ is subject to the following bounds
\begin{equation}
   \frac{\pi^2}{3N^2}-\frac{\pi^4}{45N^4}\left(2+45M^2\right)
   <
   \frac{A^2_{pqr}-{\bar A}^2_{pqr}}{A^2_{pqr}}
   <
   \frac{\pi^2}{3N^2}
\end{equation}
But since the total area of the triangulation ${\bar S}$ is given by
${\bar S}=4NM{\bar A}_{pqr}$ while the total area of the cylinder is
$S=4NMA_{pqr} = 2\pi$ we see that
\begin{equation}
   \frac{4\pi^4}{3N^2}-\frac{4\pi^6}{45N^4}\left(2+45M^2\right)
   <
   S^2-{\bar S}^2
   <
   \frac{4\pi^4}{3N^2}
\end{equation}
This clearly shows that for the total error to vanish we need not only $N\rightarrow\infty$
but also $(M/N^2)\rightarrow0$. If these two conditions are not satisfied then the total
area of the triangulation need not converge to that of the cylinder (and can even diverge to
infinity).

We will now repeat the above calculation but this time using an adjusted set of $L^2_{ij}$
given by (\ref{eqn:corollary1}). The normal vector for any point on a cylinder of unit
radius is easily computed. For example, for any point on the cylinder with coordinates
$(x,y,z)$ the unit normal vector at that point has components $(x,y,0)$. This allows us to
easily apply equation (\ref{eqn:corollary1}) to estimate the ${\tilde L}_{ij}$ on the
cylinder. Note also that since the cylinder has zero Gaussian curvature we can use the
standard Euclidian formula for the area of triangle. Thus we have
\begin{align}
   {\tilde L}^2_{pr} &= {\bar L}^2_{pr}
                      + \frac{1}{12}\left(n_\mu(p)-n_\mu(r))\Delta x^\mu_{pr}\right)^2\\[5pt]
   {\tilde L}^2_{pq} &= {\bar L}^2_{pq}
                      + \frac{1}{12}\left(n_\mu(p)-n_\mu(q))\Delta x^\mu_{pq}\right)^2\\[5pt]
   {\tilde A}^2_{pqr} &= \frac {1}{16}{\tilde L}^2_{pr}\left(4{\tilde L}^2_{pq}
                                                             -{\tilde L}^2_{pr}\right)
\end{align}
With these estimates for ${\tilde L}^2_{ij}$ we find the following error bounds for the total
area
\begin{equation}
   \frac{32\pi^6}{45N^4}-\frac{8\pi^8}{189N^6}\left(6+63M^2\right)
   <
   S^2-{\tilde S}^2
   <
   \frac{32\pi^6}{45N^4}
\end{equation}
where ${\tilde S}$ is the total area of the triangulation (using the adjusted arc-lengths).
As with the previous example, ${\tilde S}$ will converge to $S$ only when
$N\rightarrow\infty$ while $M/N^2\rightarrow0$. But note that when $M/N^2\rightarrow0$ the
errors bounds are of order $\BigO{N^{-4}}$ whereas in the previous example (using Euclidian
arc-lengths) the errors were of order $\BigO{N^{-2}}$. This is a considerable improvement.


\providecommand{\href}[2]{#2}\begingroup\raggedright\endgroup

\end{document}